\documentclass[10pt,twoside]{amsart}

\usepackage{amsfonts}
\usepackage{amsmath}
\usepackage{mathptmx}
\usepackage[all]{xy}
\usepackage{latexsym}

\newtheorem{thm}{{\sc Theorem}}[section]
\newtheorem{prop}[thm]{{\sc Proposition}}
\newtheorem{lem}[thm]{{\sc Lemma}}
\newtheorem{cor}[thm]{{\sc Corollary}}

\newtheorem{claim}[thm]{{\sc Claim}}

\newtheorem{remark}[thm]{{\sc Remark}}
\theoremstyle{definition}

\newcommand{\Deform}[0]{{\operatorname{Def}}}

\title{On deformations of Lagrangian fibrations}
\author{Daisuke Matsushita}
\address{Division of Mathematics, Graduate School of Science,
         Hokkaido University,  Sapporo, 060-0810 Japan}
\thanks{* Partially supported by Grant-in-Aid \# 15740002
 (Japan Society for Promotion of Sciences).} 
\email{matusita@math.sci.hokudai.ac.jp}

\begin{document}

\begin{abstract}
 Let $X$ be an irreducible symplectic manifold
 and $\Deform (X)$ the Kuranishi family.
 Assume that $X$ admits a Lagrangian fibration.
 We prove that $X$ can be deformed preserving
 a Lagrangian fibration. More precisely,
 there exists a smooth hypersurface $H$
 of $\Deform (X)$ such that
 the restriction family $\mathcal{X}\times_{\Deform (X)}H$
 admits a family of
 Lagrangian fibrations over $H$.
\end{abstract}

\maketitle

\section{Introduction}
 A compact K\"ahler manifold $X$ is said to be
 {\em symplectic\/} if $X$ carries a holomorphic symplectic form.
 Moreover $X$ is said to be
 {\em irreducible symplectic\/} if $X$ satisfies the following
 two properties:
\begin{enumerate}
 \item  $\dim H^0 (X, \Omega_X^2) = 1$ and;
 \item  $\pi_1 (X) = \{1\}$.
\end{enumerate}
 A surjective morphism between K\"ahler spaces
 is said to be {\em fibration} if 
 it is surjective and 
 has only connected fibres.
 A fibration from a symplectic manifold
 is said to be {\em Lagrangian\/} 
 if a general fibre is a Lagrangian submanifold.
 An example of an irreducible symplectic
 is a $K3$ surface. An elliptic fibration from a $K3$ surface
 gives an example of a Lagrangian fibration.
 It is expected that
 a $K3$ surface and an irreducible symplectic manifold
 share many geometric properties.
 Let $S$ be a $K3$ surface and $g : S \to \mathbb{P}^1$
 an elliptic fibration. 
 Kodaira proves that
 there exists a smooth hypersurface $H_S$ in 
 the Kuranishi space $\Deform (S)$ of $S$
 which has the following three properties:
\begin{enumerate}
 \item The hypersurface $H_S$ passes the reference point.
 \item For  the Kuranishi family $\mathcal{S}$ of $S$,
       there exists an open neighbourhood $ U_{S} $ of the
       reference point in $ H_{S} $ such that
       the base change $\mathcal{S}\times_{\Deform (S)}U_S$
       admits a surjective morphism 
       $\mathcal{S}\times_{\Deform (S)}U_{S} \to \mathbb{P}^1_{U_S}$ which satisfies
       the following diagram:
$$
 \xymatrix{
 \mathcal{S}\times_{\Deform (S)}U_S \ar[r] \ar[d]
 & \mathbb{P}^1_{U_S} \ar[dl] \\
 U_S. & 
 }
$$
 \item The original fibration $g$ coincides with
       the restriction of the above diagram over the reference
       point. For every point $t$ of $U_S$,
       the restriction of  the diagram over $t$
       gives an elliptic fibration from $\mathcal{S}_t$, which is the
       fibre over $t$.
\end{enumerate}
 In this note, we prove a higher dimensional analog of 
 the above statement. To state the main result, we need a description
 of the local universal deformation of a pair of
 an irreducible symplectic manifold and a line bundle
 by
 Beauville \cite[Th\'eor\`em 5 and Corollaire 1]{beauville}.
\begin{thm}\label{local_torelli}
  Let $X$ be an irreducible symplectic manifold,
  $L$ a line bundle on $X$ and $\mathrm{Def} (X)$
 the Kuranishi space of $X$. We denote by $q_X$ the Beauville-Bogomolov-Fujiki
 quadratic form on $H^2(X,\mathbb{C})$. Let $Q$ be the open set of
 the quadratic hypersurface in 
 $\mathbb{P}(H^2(X,\mathbb{C}))$ defined by
$$
 Q := \{
 \alpha  \in H^2 (X,\mathbb{C}) ;
 q_X(\alpha) = 0, q_X(\alpha + \bar{\alpha})> 0
 \}.
$$
 Then the image of the period map 
$$
 p : \mathrm{Def}(X) \to \mathbb{P}(H^2(X,\mathbb{C})).
$$
 is contained in $Q$ and
 locally isomorphic in a neighbourhood of the reference
 point of $\mathrm{Def}(X)$. Let $H$ be the preimage of the
 intersection of $\mathbb{P}(L^{\perp})$ and $Q$ by $p$,
 where $L^{\perp}$ is the orthogonal space of $L$ in 
 $H^2(X,\mathbb{C})$ with respect to $q_X$. There
 exists a line bundle $\mathcal{L}$ on $\mathcal{X}\times_{\Deform(X)}H$
 such that $\mathcal{L}|_X =L$. The pair
 $(\mathcal{X}\times_{\Deform(X)}H, \mathcal{L})$ forms the local
 universal deformation family of the pair $(X,L)$.
\end{thm}
 The following is the main theorem.
\begin{thm}\label{main}
 Let $X$ be an irreducible holomorphic symplectic manifold
 and $\mathcal{X} \to \Deform (X)$ the Kuranishi family of $X$.
 Assume that $X$ admits a Lagrangian fibration
 $f : X \to B$ over a projective variety $B$.
 Let $L$ be the pull back of
 an ample line bundle  on $B$.
 We also let  $\mathcal{L}$ and $H$ be
 as in Theorem \ref{local_torelli}. Then
$
 R^i\pi_*\mathcal{L}
$
 is locally free for every $i$ on an open neighbourhood of 
 the reference point in $ H $,
  where $\pi$ is the projection
 $\mathcal{X}\times_{\Deform(X)}H \to H$
\end{thm}

\begin{cor}\label{commentary}
       Let $f : X \to B$ be as in Theorem \ref{main}.
       We also let
       $L$ be the pull back of a very ample line buncle
       of $B$. The symbols 
       $\pi$, $\mathcal{X}$, $H$ and $\mathcal{L}$ denote
       same objects
       as in Theorem \ref{main}.
       Then $\mathcal{L}$ is $\pi$-free
       in a small neighbourhood of the reference point, that is,
       there exists an open neighbourhood $ U $ of the reference point
       in $ H $
       and a morphism 
       $f_U : \mathcal{X}\times_{\Deform (X)}U  \to
       \mathbb{P}(\pi_* \mathcal{L})|_{U}$. Together with 
       $\pi$, they form
       the following diagram:
$$
 \xymatrix{
 \mathcal{X}\times_{\Deform (X)}U \ar[r]^{f_U} \ar[d]_{\pi}
 & \mathbb{P}({\pi_{*}\mathcal{L}})|_{U} \ar[dl] \\
 U, & 
 }
$$
       The original fibration $f$ coincides with
       the restriction of the above diagram over the reference
       point. For every point $t$ of $U$,
       the restriction of 
       the diagram over $t$
       gives a Lagrangian fibration from $\mathcal{X}_t$, which is the fibre over $t$.
\end{cor}

\begin{remark}
 Hassett and Tschinkel obtained 
 Corollary \ref{commentary} in 
 \cite[Theorem 4.4]{hassett} under the assumptions that
 $X$ is deformation equivalent to an irreducible component of the
 Hilbert scheme of a $K3$ surface which represents
 length two subschemes
 and the higher cohomologies of
 $L$ vanishes.
\end{remark}

\begin{remark}
 If $X$ is an irreducible symplectic manifold.
 Assume that $X$ admits a surjective morphism 
 $f : X \to B$ such that $f$ has connected fibres and
 $0 < \dim B < \dim X$.
 If $X$ and $B$ are projective or $X$ and $B$ are smooth and K\"ahler,
 then $f$ is Lagrangian and $B$ is projective by 
 \cite{matsu}, \cite{matsu2} and \cite[Proposition 24.8]{GJK}.
\end{remark}

\section*{Acknowledgement}
 The author would like to express his thanks to Professors
 A.~Fujiki, E.~Markman,
 O.~Debarre, 
 Jun-Muk. ~Hwang, N.~Nakayama
 and K.~Oguiso
 for their comments.

\section{Proof of Theorem}
 Before proving Theorem \ref{main}, 
 we prepare two Propositions. 

\begin{prop}\label{first_step}
 Let $f : X \to B$, $\mathcal{X} \to \Deform (X)$, $ H $ 
 and $L$ be as in Theorem \ref{main}.
 We denote by $A$ a general fibre of $f$.
 Then there exists 
 a smooth torus fibration $ \mathcal{A} \to H $ which satisfy the following 
 diagram:
 \[
 \xymatrix{
 \mathcal{X} \ar[d] &  \mathcal{A} \ar[d]^{\phi} \ar[l]_{ev} \\
 \Deform (X) & H \ar[l]^{j}
 }
 \]
 where $ j $ is the natural inclusion. For each point $ u $ of $ H $,
  $ ev(\phi^{-1}(u)) $ defines a Lagrangian torus in $ \mathcal{X}_{j(u)} $.
 
\end{prop}

\begin{proof}[Proof of Proposition \ref{first_step}.]
 We need the following Lemma.
\begin{lem}\label{key_lemma}
 Let $X$, $L$ and $A$ be as in Proposition \ref{first_step}.
  For an element $z$ of $H^{2}(X, \mathbb{C}),$
 the restriction 
 $z |_{A} = 0$ in $H^{2}(A,\mathbb{C})$
 if $q_X(z , L) = 0$.
\end{lem}
\begin{proof}
 Let $\sigma$ be a K\"ahler class of $X$. It is enough to prove
 that
$$
  z \sigma^{n-1} L^n = z^2 \sigma^{n-2} L^{n} = 0,
$$
 where $2n = \dim X$.
 By \cite[Theorem 4.7]{fujiki}, we have the following equation;
\begin{equation}\label{Fujiki_relation}
 c_X q_X (z + s\sigma + tL)^n = 
      (z + s\sigma + tL)^{2n},
\end{equation}
 where $s$ and $t$ are indeterminable numbers and  
 $c_X$ is a constant only depending on $X$.
 By the assumption,
$$
 c_X q_X(z + s\sigma + tL)^n =
 c_X (q_X(z) + s^2q_X(\sigma) + 2sq_X(z, \sigma)
 + 2stq_X(\sigma, L))^n .
$$
 If we compare the $s^{n-1}t^n$ 
 and $s^{n-2}t^n$ terms of 
 the both hand sides of the equation (\ref{Fujiki_relation}),
 we obtain the assertions.
\end{proof}
 We go back to the proof of the assertion
 of Proposition \ref{first_step}.
 Let $j : H^2(X,\mathbb{C}) \to H^2(A,\mathbb{C})$ 
 be the natural induced morphism by the inclusion
 $A \to X$.
 We also  let $H_A$ be 
 the preimage of 
$$
  Q \cap \mathbb{P}({\mathrm{Ker}}(j)),
$$
 by the period map of $ p $.
 By \cite[0.1 Theorem]{voisin},
 there exists a smooth torus fibration $ \mathcal{A} \to H_{A} $ which
 satisfies the following diagram:
 \[
 \xymatrix{
 \mathcal{X} \ar[d] & \mathcal{A} \ar[l]_{ev} \ar[d]^{\phi} \\
 \Deform (X) & H_{A} \ar[l]^{j}.
 }
 \]
 where $ j $ is the natural inclusion. For each point $ u $ of $ H_{A} $,
 $ ev (\phi^{-1}(u)) $ defines a Lagrangian torus in $ \mathcal{X}_{j(u)} $.
 By Lemma \ref{key_lemma}, 
 $L^{\perp} \subset {\mathrm{Ker}}(j)$. Since
 the restriction of K\"ahler classes of $X$ defines
 K\"ahler classes of $A$,
 we obtain the following
 inequalities:
$$
 \dim L^{\perp} \le \dim {\mathrm{Ker}}(j) \le \dim H^2 (X,\mathbb{C}) - 1.
$$
 Hence we obtain
 ${\mathrm{Ker}}(j) = L^{\perp}$. This implies that $ H_{A} = H $.
\end{proof}

\begin{prop}\label{one_dimensional_cutting}
 Let $\mathcal{X}$, $\Deform (X)$, $\mathcal{L}$
  and $H$ be as in Theorem \ref{main}.
 We also let $\Delta$ be a unit disk in $H$ which
 passes the reference point of $\Deform(X)$.
 The symbols
 $\mathcal{X}_{\Delta}$,  $\pi_{\Delta}$
 and $\mathcal{L}_{\Delta}$ denote
 the base change $\mathcal{X}\times_H \Delta$, 
 the induced morphism $\mathcal{X}_{\Delta} \to \Delta$
 and the restriction $\mathcal{L}$ to $\mathcal{X}_{\Delta}$,
 respectively.
 Assume that the Picard number of 
 the fibre $\mathcal{X}_t$ of $\pi_{\Delta}$
 over $t$ is one for 
 a very general point $t$ of $\Delta$. 
 Then
$$
 R^i (\pi_{\Delta})_* \mathcal{L}_{\Delta}
$$
 are locally free for all $i$ at the reference point.
\end{prop}

\begin{proof}
 For a point $u$ of $\Delta$, 
 $\mathcal{X}_u$ denote 
 the fibre of $\pi_{\Delta}$ over $u$
 and $\mathcal{L}_u$ denote
 the restriction
 of $\mathcal{L}_{\Delta}$ to $\mathcal{X}_u$.
 We also denote by $o$ the reference point.
 If there exists an open neighbourhood $V$ of $o$ in \( \Delta\) such that
 $\mathcal{L}_u$ is semi-ample for every point $u$ of $V$ except $o$,
 the assertion
 of Proposition \ref{one_dimensional_cutting}
 follows from \cite[Corollary 3.14]{nakayama}.
 We prove it by the following Lemma \ref{very_general} and Lemma 
 \ref{very_general=>open}.
\begin{lem}\label{very_general}
 We use the same notation as in Proposition \ref{one_dimensional_cutting}.
 Let $\mathcal{X}_u$ be a fibre of $\pi_{\Delta}$ whose Picard number 
 is one. Then
 $\mathcal{L}_u$ is semi-ample.
\end{lem}
 The proof of Lemma \ref{very_general} will be given after the following
 two Claims.
\begin{claim}\label{nonprojective}
  Under the hypothesis of Lemma \ref{very_general},
 \begin{enumerate}
  \item $\mathcal{X}_u$ is not projective and;
  \item $\mathcal{L}_u$ is nef.
 \end{enumerate}
\end{claim}

\begin{proof}
 (1) We denote by $q_{\mathcal{X}_u}$ the Beauville-Bogomolov-Fujiki
 quadratic 
 form on $\mathcal{X}_u$.
 Since $(\mathcal{X}_u , \mathcal{L}_u)$ and $(X,L)$
 are deformation equivalent, 
 $q_X (L) = q_{\mathcal{X}_u}(\mathcal{L}_u) = 0$.
 By the assumption, the Picard number of $\mathcal{X}_u$ is one.
 Hence
 $q_{\mathcal{X}_u}(z) = 0$ for every element
 of $H^{1,1} (\mathcal{X}_u,\mathbb{C})_{\mathbb{Q}}$.
 By \cite[Corollary 3.8]{basic_result},
 $\mathcal{X}_u$ is not projective.
\newline
 (2)  This follows from  \cite[3.4 Theorem]{oguiso}. For the convenience of readers,
 we copy their arguments.
 By \cite[Proposition 3.2]{kahler_cone}
 it is enough to prove that 
 $\mathcal{L}_u. C \ge 0$ for every effective curve of $\mathcal{X}_u$.
 Since $q_{\mathcal{X}_u}$ 
 is non-degenerate
 and defined over $H^2 (\mathcal{X}_u, \mathbb{Q})$,
 there exists an isomorphic 
$$
 \iota : H^{1,1}(\mathcal{X}_u,\mathbb{C})_{\mathbb{Q}} \to
        H^{2n-1,2n-1}(\mathcal{X}_u,\mathbb{C})_{\mathbb{Q}}
$$
 such that 
$$
 q_{\mathcal{X}_u}(\mathcal{L}_u , \iota^{-1}([C])) = \mathcal{L}_u.C.
$$
 If $q_{\mathcal{X}_u}(\mathcal{L}_u, z) \ne 0$ for an element $z$ of 
 $H^{1,1}(\mathcal{X}_u,\mathbb{C})_{\mathbb{Q}}$,
 then there exists a rational number $d$ such that
 $q_{\mathcal{X}_u}(\mathcal{L}_u + dz) > 0$.
 By \cite[Corollary 3.9]{basic_result},
 this implies that $\mathcal{X}_u$ is projective.
 That is a contradiction. Thus
 $\mathcal{L}_u .C = 0$ for every curve $C$.
\end{proof}

\begin{claim}\label{almost_holomorphic}
 Under the hypothesis Lemma \ref{very_general}, there exists a 
 dominant meromorphic map $\Phi : \mathcal{X}_u \dasharrow B_u$ such that
  a general fibre of $\Phi$ is compact, $B_u$ is a K\"ahler manifold and
 $\dim B_u > 0$.
\end{claim}
\begin{proof}
 We use the notation as in Proposition \ref{first_step}.
 Since $ u \in H $,
 $ \mathfrak{X}_{u} $ contains a Lagrangian torus $ \mathcal{A}_{u} $
 by Proposition \ref{first_step}.
 Let $D(\mathcal{X}_u)$ 
 be the irreducible component of the Douady space of $ \mathcal{X}_{u} $
 which
 contains a point corresponding to $\mathcal{A}_{u}$
 and $D(\mathcal{X}_u)^{-}$ 
 a resolution of $D(\mathcal{X}_u)$.
 We note that the Douady space of $ \mathcal{X}_{u} $ is smooth at the point
 corresponding to $ \mathcal{A}_{u} $ by \cite[Theorem 2.2]{ran_deformation_map}.
 We
 denote by $U(\mathcal{X}_u)^{-}$ the normalization
 of $U(\mathcal{X}_u)\times_{D(\mathcal{X}_u)}D(\mathcal{X}_u)^{-}$,
 where $U(\mathcal{X}_u)$ is
 the universal family
 over $D(\mathcal{X}_u)$. We also denote 
 by $p_1$ and $p_2$ the natural projections 
 $U(\mathcal{X}_u)^{-} \to \mathcal{X}_u$ and
 $U(\mathcal{X}_u)^{-} \to D(\mathcal{X}_u)^{-}$.
 The relations of these objects are summarized  in the following
 diagram:
$$
 \xymatrix{
 \mathcal{X}_u & U(\mathcal{X}_u)     \ar[l] \ar[d] 
               & U(\mathcal{X}_u)^{-} \ar[l] \ar[d]^{p_2} 
                 \ar@/_1pc/[ll]_{p_1}  \\
               & D(\mathcal{X}_u)
               & D(\mathcal{X}_u)^{-} \ar[l].
 }
$$
 Let $a$ be a point of $\mathcal{X}_u$.
 We define the subvarieties $G_i (a)$ of $\mathcal{X}_u$ by
\begin{eqnarray*}
 && G_0 (a):= a   \\
 && G_{i+1} (a) := p_1(p_2^{-1}(p_2(p_1^{-1}(G_{i}(a)))))
\end{eqnarray*}
 We also define 
$$
 G_{\infty}(a) := \bigcup_{i = 0}^{\infty} G_i (a).
$$
 Let $B(\mathcal{X}_u)$ be the Barlet space of $\mathcal{X}_u$.
 By \cite[Th\'eor\`eme A.3]{campana}, $G_{\infty}(a)$ is 
 compact for a general point $a$ of $\mathcal{X}_u$.
 We define a meromorphic map by
 \[
 \Phi : \mathcal{X}_u \ni a \mapsto G_{\infty}(a) \in  B(\mathcal{X}_u).
 \]
 By \cite[(5.2) Theorem]{fujiki}, $B(\mathcal{X}_u)$ is of class $\mathcal{C}$.
 Hence there exists an embedded resolution 
 $B(\mathcal{X}_u)^{\sim} \to B(\mathcal{X}_u)$
 of the closure of the image of $\Phi$
 whose proper transformation is smooth and K\"ahler.
 We denote by $B_u$ the proper transformation. Then we obtain a meromorphic map
$$
 \mathcal{X}_u \dasharrow B_u ,
$$
 which is the desired one if
 $B_{u}$ is not a point.
 Hence we show that $G_{\infty}(a)$ is not equal to $\mathcal{X}_u$
 for a general point $a$ of $\mathcal{X}_u$.
 Let $F$ be a general fiber of $p_2$.
 Since $p_1$ is the natural projection
 from the base change of the universal family
 $U(\mathcal{X}_u)\times_{D(\mathcal{X}_u)}D(\mathcal{X}_u)^-$
 to $\mathcal{X}_u$,
 $p_1(F)$ defines a Lagrangian torus of
 $\mathcal{X}_u$. By \cite[Theorem
 2.2]{ran_deformation_map} $p_1(F)$ is unobstructed.
 The normal bundle of $p_1 (F)$
 is isomorphic to the tangent bundle of $F$,
 Hence
 $p_1$ is locally isomorphic in a neighbourhood of $p_1 (F)$ and 
 generically finite.
 If $p_1$ is bimeromorphic, then $G_{\infty}(a) = G_1 (a)$ and we are
 done.
 If $p_1$ is not bimeromorphic, we consider the branch locus of 
 the Stein factorization of $p_1$.
 If the branch locus is empty, then $U(\mathcal{X}_u)^{-}$
 is not irreducible, because $\mathcal{X}_u$ is simply connected.
 That is a contradiction. Thus the branch locus is non empty.
 Since $\mathcal{X}_u$ is smooth, the branch locus
 defines 
 an effective divisor $E$ of $\mathcal{X}_u$.
 If we prove that there exists an effective $\mathbb{Q}$-divisor
 $E'$ such that
$$
 p_1^*E = p_2^*E'
$$
 then $G_{\infty} (a) \cap E = \emptyset$ for a point $a \not\in E$
 and we are done. 
 By \cite[Lemma 2.15]{nakayama}, we need to show that
 $E$ is nef and $p_2(p_1^{-1}(E)) \ne D(\mathcal{X}_u)^{-}$.
 Since the Picard number of $\mathcal{X}_u$ is one, 
 $\mathcal{L}_u$ and $\pm E$ should be numerically proportional.
 By Claim \ref{nonprojective}, 
 $E$ is nef, because
 $E$ is effective. 
 Since $p_1$ is isomorphic in a neighbourhood of $F$,
 $E \cap F = \emptyset$. This implies that 
 $p_2(p_1^{-1}(E)) \ne D(\mathcal{X}_u)^{-}$.
\end{proof}

\begin{proof}[Proof of Lemma \ref{very_general}]
 By Claim \ref{almost_holomorphic}, there
 exists a dominant meromorphic map $\Phi : \mathcal{X}_u \dasharrow B_u$
 whose general fibre is compact.
 By blowing ups and flattening, we have the following diagram:
$$
 \xymatrix{
 \mathcal{X}_u \ar@{.>}[d]& \mathcal{Y}_u \ar[l]  \ar[d]
 & \mathcal{Z}_u \ar[l] \ar[d] & \mathcal{W}_u \ar[l] \ar[d]^{r} 
 \ar@/_6pt/[lll]_{\nu}
 & \mathcal{W}_{u}^{\sim} \ar[l] \ar[d]^{r^{\sim}}
 \ar@/_12pt/[llll]_{\nu^{\sim}} \\
 B_u & B_u \ar@{=}[l] & B^{\sim}_u \ar[l]  & B^{\sim}_u, \ar@{=}[l]
 & B^{\sim}_{u} \ar@{=}[l]
 }
$$
 where 
\begin{enumerate}
 \item $\mathcal{Y}_u \to \mathcal{X}_u$ is a resolution of
       indeterminacy of $\Phi$.
 \item $\mathcal{Z}_u \to \mathcal{Y}_u$ and
       $B^{\sim}_u \to B_u$ are bimeromorphic.
 \item $B^{\sim}_u$ is smooth and K\"ahler.
 \item $\mathcal{Z}_u \to B^{\sim}_u$ is flat.
 \item $\mathcal{W}_u \to \mathcal{Z}_u$ is the normalization.
 \item \( \mathcal{W}^{\sim}_{u} \to \mathcal{W}_{u}\) is a resolution
\end{enumerate}
 We denote by $\nu$, $ \nu^{\sim} $, $r$ and $ r^{\sim} $ the induced morphisms 
 $\mathcal{W}_u \to \mathcal{X}_u$, $ \mathcal{W}^{\sim}_{u} \to \mathcal{X}_{u} $,
 $ \mathcal{W}_{u} \to B_{u}^{\sim} $
 and $\mathcal{W}^{\sim}_u \to B^{\sim}_u$, respectively.
 We note that $ \nu $ and $ \nu^{\sim} $ are isomorphic on general fibres
 of $ \Phi $.
 The proof consists of three steps.

\noindent
{\bfseries Step 1. \quad}
 We prove that $B^{\sim}_u$ is projective. 
 Since $B^{\sim}_u$ is K\"ahler,
 it is enough to prove that $\dim H^0 (B^{\sim}_u, \Omega^2) = 0$.
 We derive a contradiction assuming that
 $\dim H^0 (B^{\sim}_u, \Omega^2) > 0$.
 Under this assumption, there exists a holomorphic $2$-form $\omega$
 on $B^{\sim}_u$. The pull back $(r^{\sim})^* \omega$ defines
 a degenerate holomorphic $2$-form on $\mathcal{W}^{\sim}_u$.
 On the other hand, 
 $H^0 (\mathcal{W}^{\sim}_u , \Omega^2) \cong H^0 (\mathcal{X}_u, \Omega^2)$
 because $\nu^{\sim}$ is birational and $\mathcal{X}_u$ and $ \mathcal{W}_{u} $ are smooth.
 Hence $\dim H^0 (\mathcal{W}^{\sim}_u, \Omega^2) = 1$
 and it should be generated by a generically nondegenerate holomorphic
 $2$-form. That is a contradiction.
 
\noindent
{\bfseries Step 2. \quad}
 Let $M$ be a very ample divisor on $B^{\sim}_u$. We prove that
 there exists a rational number $c$ such that
$$
\mathcal{L}_u \sim_{\mathbb{Q}}  c \nu_* r^*M.
$$
 It is enough to prove that
$$
 q_{\mathcal{X}_u}(\nu_* r^*M ) = 
 q_{\mathcal{X}_u}(\mathcal{L}_u) = 
 q_{\mathcal{X}_u}(\nu_*r^* M , \mathcal{L}_u) = 0.
$$
 Since $\mathcal{X}_u$ is non projective,
 $q_{\mathcal{X}_u}(\nu_* r^*M) \le 0$ and 
 $q_{\mathcal{X}_u}(\mathcal{L}_u ) \le 0$
 by \cite[Corollary 3.8]{basic_result}.
 On the other hand,  
 $q_{\mathcal{X}_u}(\mathcal{L}_u ) \ge 0$ 
 because $\mathcal{L}_u$ is nef.
 The linear system $|r^*M|$ contains
 members $M_1$ and $M_2$ such that the codimension of
 $M_1 \cap M_2$ is two.
 By the definition
$$ 
 q_{\mathcal{X}_u}(\nu_* r^*M) = \frac{\dim \mathcal{X}_u}{2}
 \int (\nu_* r^*M)^2 \sigma^{n-1}
 \bar{\sigma}^{n-1},
$$
 where $\sigma$ is a symplectic form on $\mathcal{X}_u$.
 Thus $q_{\mathcal{X}_u}(\nu_* r^*M) \ge 0$. Therefore 
 $q_{\mathcal{X}_u}(\nu_* r^*M ) = q_{\mathcal{X}_u}(\mathcal{L}_u) = 0$.
 Since $\nu_*r^*M$ is effective and $\mathcal{L}_u$ is nef,
 $q_{\mathcal{X}_u}(\nu_* r^*M , \mathcal{L}_u) \ge 0$.
 Again by  \cite[Corollary 3.8]{basic_result},
 $q_{\mathcal{X}_u}(\nu_*r^*M + \mathcal{L}_u) \le 0$.
 Thus $q_{\mathcal{X}_u}(\nu_* r^*M , \mathcal{L}_u ) = 0$ and we are done.

\noindent
{\bfseries Step 3. \quad}
 We prove that $\mathcal{L}_u$ is semi-ample.
 By \cite[Remark 2.11.1]{nakayama} and
 \cite[Theorem 5.5]{nakayama}, 
 it is enough to prove that there exists a nef and big divisor $M'$ on
 $B^{\sim}_u$ such that
 $$
 \nu^{*}\mathcal{L}_u \sim_{\mathbb{Q}} r^*M'
 $$
 By Step 2, $ |\mathcal{L}_{u}^{\otimes k}| \ne \emptyset $ for some positive integer
  $ k  $. Hence
 $
 \nu^*\mathcal{L}_u \sim_{\mathbb{Q}} r^*M + F
 $
 where  $F$ is a $ \mathbb{Q} $-effective Cartier divisor whose support
 is contained in the exceptional locus.
 For an irreducible component $ \Gamma $
 of a fibre of $r$, we have
 \[
  \nu^{*}\mathcal{L}_{u} |_{\Gamma} = F|_{\Gamma}.
 \]
 Thus the restriction 
 $ F|_{\Gamma} $ is nef for all $ \Gamma $. 
 Moreover $ r(F) \ne B_{u}^{\sim}$ because
 $ \nu $ is isomorphic on a general fibre of $ \Phi $. 
 By \cite[Lemma 2.15]{nakayama},
 there exists a $\mathbb{Q}$-effective divisor $M_0$ such that
$$
 F = r^* M_0.
$$
 If we put $M' := M + M_0$, we complete the proof of Lemma \ref{very_general}.

\end{proof}

\begin{lem}\label{very_general=>open}
 Let $\mathcal{X}_{\Delta}$, $\mathcal{L}_{\Delta}$, $\Delta$,
 $\mathcal{X}_u$ and $\mathcal{L}_u$
 be as in Proposition \ref{one_dimensional_cutting} 
 and the first part of its proof.
 If $\mathcal{L}_u$ is semi-ample a for very general point $u$
 of $\Delta$, then there exists an open neighbourhood $V$ of $o$ such that
 $\mathcal{L}_u$ is semi-ample
 for every point $u$ of $V$ except $o$.
\end{lem}
\begin{proof}
Let $ \varphi_{k} $ be the function on $ \Delta $ defined by
\[
 \varphi_{k}(t) := \dim 
 H^{0}(\mathcal{X}_{u}, \mathcal{L}^{\otimes k}_{\Delta}|_{\mathcal{X}_{u}}).
\]
 The function $ \varphi_{k} $ is upper semi continuous. Thus we have an open set
 of $ \Delta $ such that $ \varphi_{k} $ is constant. We denote by
  $\Delta (k)$ this open set. Then 
$$
 \mathcal{L}_{\Delta}^{\otimes k}\otimes k(u) \cong
 H^0 (\mathcal{X}_u , \mathcal{L}_{\Delta}^{\otimes k}|_{\mathcal{X}_u}).
$$
for $ u \in \Delta (k) $.
 By the assumption,
there exists a point $u_0$ of 
$\cap_{k=1}^{\infty}\Delta (k)$ such that 
 $\mathcal{L}_{u_0}$ is semiample. Hence
$$
 \pi_{\Delta}^* (\pi_{\Delta})_{*} \mathcal{L}_{\Delta}^{\otimes k}
 \to \mathcal{L}_{\Delta}^{\otimes k}
$$
 is surjective in a neighbourhood of  $\mathcal{X}_{u_0}$ for some $k$.
 This implies that the support $Z$ of the cokernel sheaf of
$\pi_{\Delta}^* (\pi_{\Delta})_{*} \mathcal{L}_{\Delta}^{\otimes k} 
 \to \mathcal{L}_{\Delta}^{\otimes k}$
 is a proper closed subset of $\mathcal{X}_{\Delta}$. 
 Since $\pi_{\Delta}$ is proper, $\pi (Z)$ is also a proper closed
  subset of $\Delta$.
 If we put $V = \Delta \setminus \pi (Z)$,
 we are done.
\end{proof}
 We complete the proof of Proposition \ref{one_dimensional_cutting}.
\end{proof}

\begin{proof}[Proof of Theorem \ref{main}.]
 We define the function $\varphi (t)$ on $ H $ as
$$
\varphi (t) :=
  \dim H^i (\mathcal{X}_t, \mathcal{L}_t)
$$
 where $\mathcal{X}_t$ is the fibre of $\pi$
 over $t$ and $\mathcal{L}_t$ is the restriction
 of $\mathcal{L}$ to $\mathcal{X}_t$.
 If $ \varphi (t) $ is constant on an open neighbourhood $ U $ of 
 the reference point,
  $ R^{i}\pi_{*}\mathcal{L} $ is locally free on $ U $.
 Let $ t $ be a very general point of $ H $. We choose 
  a small disk $ \Delta $ in $ H $ such that $ \Delta $ passes 
  $ o $ and $ t $.
 Since the Picard number of a 
 fibre $\mathcal{X}_t$ over a very general point
 of $H$ is one,
 the Picard number of a very general fibre
 of the induced morphism $\mathcal{X}\times_H \Delta \to \Delta$
 is one.
 By Proposition \ref{one_dimensional_cutting},
 $R^i (\pi_{\Delta})_* \mathcal{L}_{\Delta}$
 is locally free for every $i$. By the criteria
 of cohomological flatness \cite[page 134]{banica},
 if $ R^{i}(\pi_{\Delta})_{*}\mathcal{L}_{\Delta} $ is locally free and the morphism
 \begin{equation}\label{basechange}
   R^{i} (\pi_{\Delta})_{*}\mathcal{L}_{\Delta} \otimes k(t) \cong 
   H^{i}(\mathcal{X}_{t},\mathcal{L}_{\Delta}|_{\mathcal{X}_{t}})
 \end{equation}
 is isomorphic, 
 we have an isomorphism
 \[
 R^{i-1} (\pi_{\Delta})_{*}\mathcal{L}_{\Delta} \otimes k(t) \cong 
   H^{i-1}(\mathcal{X}_{t},\mathcal{L}_{\Delta}|_{\mathcal{X}_{t}})
 \]
 If $ i \ge \dim X $, the both hand sides of (\ref{basechange}) are zero. By a reverse
 induction and Proposition \ref{one_dimensional_cutting},
  we have that the morphisms (\ref{basechange}) are isomorphic for all $ i \ge 0 $.
 By Proposition \ref{one_dimensional_cutting} and the isomorphisms (\ref{basechange}), 
  $\varphi (t)$ is constant on $ \Delta $. 
 Let $ Z $ be a subset of $ H $ such that
 \[
  Z := \{ t \in H; \varphi (t) > \varphi (o),
  \}
 \]
 where $ o $ is the reference point. By the upper semicontinuous, $ Z $ is close.
 Since very general points of $ H $ are not contained in $ Z $, $ H\setminus Z $ is
 an open neighbourhood of $ o $.
 If we put $ U = H\setminus Z $, we are done.
\end{proof}

\begin{proof}[Proof of Corollary \ref{commentary}]
 The symbols $\mathcal{X}$, $H$, $\mathcal{L}$ 
 and $\pi$ denote the same objects in Corollary \ref{commentary}.
 It is enough to prove that the natural morphism
\begin{equation}\label{pi_free}
 \pi^* \pi_* \mathcal{L} \to \mathcal{L}
\end{equation} 
 is surjective in an open neighbourhood of the reference point. 
 Since $L$ is free, the restriction morphism
$$
 H^0 (X,L)\otimes L \to L
$$
 is surjective. 
 By Theorem \ref{main} and the argument in the proof of Theorem \ref{main},
\[
 \pi_{*}\mathcal{L}\otimes k(o) \to H^{0}(X,L)
\]
is isomorphic, where $ o $ is the reference point.
 This implies that the above morphism (\ref{pi_free})
 is surjective over $X$. 
 Since surjectivity is an open condition, we are done.
\end{proof}

\end{document}